\documentclass[a4paper,12pt]{article}
\usepackage[utf8]{inputenc}
\usepackage{amssymb}
\usepackage{amsmath}
\usepackage{amsthm}
\usepackage{bbm}
\usepackage{enumitem}
\usepackage{hyperref}
\hypersetup{
	colorlinks=true,
	linkcolor=black,
	linkbordercolor={1 1 1},
	citecolor=black
}

\newtheorem{thm}{Theorem}
\newtheorem{theorem}[thm]{Theorem}
\newtheorem*{theorem*}{Theorem}

\newtheorem{lemma}[thm]{Lemma}
\newtheorem{cor}[thm]{Corollary}
\newtheorem{prop}[thm]{Proposition}

\newtheorem{problem}[thm]{Problem}

\theoremstyle{definition}
\newtheorem{definition}[thm]{Definition}
\newtheorem{remark}[thm]{Remark}

\newtheorem{example}[thm]{Example}

\numberwithin{equation}{section}
\numberwithin{thm}{section}

\newcommand{\N}{\mathbb{N}}
\newcommand{\R}{\mathbb{R}}
\newcommand{\C}{\mathbb{C}}
\newcommand{\Z}{\mathbb{Z}}

\date{}

\title{On the chain of commuting operators on Banach spaces}
\author{Tomasz Szczepanski}
\newcommand{\Addresses}{{% additional braces for segregating \footnotesize
  \bigskip
  \footnotesize

\textsc{Department of Mathematics, University of British Columbia,
    Vancouver, BC, Canada, V6T 1Z2}\par\nopagebreak
  \textit{E-mail address}: \texttt{tomasz@math.ubc.ca}

}}
\begin{document}
\maketitle

\abstract{
An operator $T$ on a Banach space is said to be of chain $N$ if there exist non-scalar operators $S_1,\dots,S_{N-1}$ and a non-zero compact operator $K$ such that 
$$T \leftrightarrow S_1 \leftrightarrow S_2 \leftrightarrow \dots\leftrightarrow S_{N-1} \leftrightarrow K,$$
where $A\leftrightarrow B$ denotes $AB=BA$. We investigate this concept by identifying classes of operators that are of chain $N$ for some $N$. Our main result establishes that every weighted shift on $\ell_p$ ($1\leq p<\infty$) is of chain $3$, which in particular includes the class of non-Lomonosov operators studied by Hadwin et al. Furthermore, we provide an example of an operator on a separable Hilbert space that cannot be connected to a compact operator via a commuting chain of any length.

}

\section{Introduction and motivation}
Throughout this paper, $X$ denotes a Banach space over the scalar field $\mathbb{F} = \R$ or $\C$, unless specified otherwise. For Banach spaces $X$ and $Y$, we let $\mathcal{L}(X,Y)$ denote the space of bounded linear operators from $X$ to $Y$, and write $\mathcal{L}(X)$ for $\mathcal{L}(X,X)$. For convenience, we define $\mathcal{N}(X)$ as the set of all non-scalar operators, that is $\mathcal{N}(X):=\{T\in \mathcal{L}(X): T\neq \lambda I \text{ for all }\lambda \in\mathbb{F}\}$. Furthermore, $\mathcal{K}(X)$ denotes the ideal of compact operators on $X$. For an operator $T\in \mathcal{L}(X)$, its commutant is denoted by $\{T\}' = \{S\in \mathcal{L}(X): TS=ST\}$. Whenever $S\in \{T\}'$, we write $T\leftrightarrow S$. 

To motivate the theory developed in this paper, we begin with the classical Invariant Subspace Problem. Recall that an operator $T\in \mathcal{L}(X)$ is said to admit a (non-trivial, proper) \textit{invariant subspace} if there exists a closed subspace $Y \subseteq X$ with $\{0\}\subsetneq Y\subsetneq X$, such that $T(Y)\subseteq Y$. 
\begin{problem}[Invariant Subspace Problem]
Given a Banach space $X$ and an operator $T\in \mathcal{L}(X)$, under what conditions does $T$ admit an invariant subspace?
\end{problem}

This topic has a rich and extensive history (see, e.g., \cite{InvariantSubspaces} for a comprehensive overview). Among the fundamental results yielding positive partial answers to the Invariant Subspace Problem, Lomonosov's Theorem \cite{Lomonosov73} stands out as one of the most celebrated.

\begin{theorem}[Lomonosov's Invariant Subspace Theorem]
\label{Lomonosov}
Let $T \in \mathcal{L}(X)$ be an operator on an infinite-dimensional complex Banach space $X$. Suppose there exist an operator $S \in \mathcal{N}(X)$ and a non-zero operator $K \in \mathcal{K}(X)$ such that 
$$T \leftrightarrow S \leftrightarrow K.$$
Then $T$ has an invariant subspace. 
\end{theorem}
We say that $T\in \mathcal{L}(X)$ is a \textit{Lomonosov operator} (shortly, $T$ is \textit{Lomonosov}) if it satisfies the assumption of Theorem \ref{Lomonosov}. Otherwise, we will call it a \textit{non-Lomonosov operator}.

At the time V. Lomonosov's work was published, it was unknown whether non-Lomonosov operators existed. An indirect answer was provided by P. Enflo \cite{Enflo} and C. Read \cite{Read1985, Read1986}, who constructed the first examples of operators without invariant subspaces which, by Theorem \ref{Lomonosov}, cannot be Lomonosov operators. Another natural question is whether every operator admitting an invariant subspace is necessarily Lomonosov. The answer to this, however, is negative.

\begin{theorem}[\cite{Hadwin}]
\label{Hadwin}
Let $H$ be a separable Hilbert space. There exists a non-empty class of operators in $\mathcal{L}(H)$ such that every operator $T$ in this class has an invariant subspace but fails to be Lomonosov.
\end{theorem}
Motivated by Theorem \ref{Lomonosov}, it is natural to ask whether longer chains of commuting operators also guarantee the existence of an invariant subspace. However, V. Troitsky \cite{Troitsky} demonstrated that this fails to hold in general.
\begin{theorem}
\label{VT}
Let $T\in \mathcal{L}(\ell_1)$ be the operator constructed in \cite{Read1986}. Then there exist operators $S_1, S_2 \in \mathcal{N}(\ell_1)$ and a rank-one operator $F\in \mathcal{L}(\ell_1)$ such that
$$T \leftrightarrow S_1 \leftrightarrow S_2 \leftrightarrow F.$$
\end{theorem}
Since the operator $T\in \mathcal{L}(\ell_1)$ constructed in \cite{Read1986} does not admit an invariant subspace, it cannot be connected to a non-zero compact operator via a shorter chain of commuting operators. Consequently, the chain in Theorem \ref{VT} is minimal for $T$. This motivates the following natural definition.

\begin{definition}
Let $N\geq 2$ be an integer. An operator $T \in \mathcal{L}(X)$ is said to be \textit{of chain $N$} if there exist operators $S_1,S_2,\dots,S_{N-1} \in \mathcal{N}(X)$ and a non-zero $K \in \mathcal{K}(X)$ such that 
$$T \leftrightarrow S_1 \leftrightarrow S_2 \leftrightarrow\dots\leftrightarrow S_{N-1} \leftrightarrow K.$$ 
Furthermore, $T$ is said to be of chain $0$ if $T\in \mathcal{K}(X)$, and of chain $1$ if it commutes with a non-zero compact operator. Finally, we define $T$ to be \textit{of minimal chain $N$} if it is of chain $N$ but not of chain $N-1$.
\end{definition}
In particular, if an operator $T\in \mathcal{L}(X)$ on an infinite-dimensional complex Banach space of chain $2$,  Theorem \ref{Lomonosov} guarantees that it admits an invariant subspace. Furthermore, Theorem \ref{Hadwin} shows that there exist operators with invariant subspaces that are not of chain $2$, while Theorem \ref{VT} demonstrates the existence of operators without invariant subspaces of minimal chain $3$. 

The goal of this note is to investigate which classes of operators are of chain $N$ for some $N\geq 0$. The first main question addressed in this work concerns the length of the chain of commuting operators for the non-Lomonosov operators from Theorem \ref{Hadwin}. In \cite{Hadwin}, it was shown that they are not of chain $2$, but it is not clear whether operators in this class can be connected to a non-zero compact operator via a longer chain. As it turns out, one extra non-scalar operator is sufficient.

\begin{theorem}
\label{non-LomonosovChain3}
Every operator $T\in \mathcal{L}(H)$ from Theorem \ref{Hadwin} is of minimal chain $3$.
\end{theorem}

This shows that the family of operators of chain $3$ is highly diverse, as it includes the operator constructed in \cite{Read1986} without invariant subspaces, as well as the non-Lomonosov operators with invariant subspaces presented in \cite{Hadwin}.

The second main result determines whether every operator can be connected via a sufficiently long chain to a non-zero compact operator. Building on an example from \cite{Ambrozie}, we show that the answer is negative.

\begin{theorem}
\label{NoCompactAmbrozie}
There exists an operator $T\in \mathcal{L}(\ell_2)$ such that for all integers $N\geq 2$, no chain of the form
$$T \leftrightarrow S_1 \leftrightarrow S_2 \leftrightarrow\dots \leftrightarrow S_{N-1} \leftrightarrow K$$
can be formed for any non-scalar operators $S_1,\dots,S_{N-1}\in \mathcal{N}(\ell_2)$ and any non-zero compact operator $K\in \mathcal{K}(\ell_2)$.
\end{theorem}

Notably, Theorem \ref{NoCompactAmbrozie} shows the existence of a non-Lomonosov operator of a different type. While the non-Lomonosov operators from \cite{Hadwin} and \cite{Read1986} can be connected to a compact operator via a chain of length $3$, the operator from \cite{Ambrozie} cannot be connected to a non-zero compact operator via any finite chain. It is worth mentioning that this operator has plenty of invariant subspaces.

The structure of this note is as follows. In Section \ref{sec:general_theory}, we provide proofs of auxiliary results showing that certain classes of operators can be connected to a rank-one operator via a finite chain. Building on the results from Section \ref{sec:general_theory}, we give a proof of Theorem \ref{Hadwin} in Section \ref{sec:non-lomonosov_shifts} by showing that a general class of weighted shifts on $\ell_p$ ($1\leq p<\infty)$, which includes the operators from Theorem \ref{Hadwin}, is of chain $3$. In Section \ref{sec:reducton_to_rank-one}, we investigate the question of connecting an operator not only to a non-zero compact operator $K$ via a chain of commuting operators, but specifically to a rank-one operator $F$. In the final section, we give a short exposition of the theory of commuting graphs and prove Theorem \ref{NoCompactAmbrozie}.

\section{General theory}
\label{sec:general_theory}

We start with the following simple, yet useful observation. It is easy to see that if an operator is non-injective or without a dense range, it has an invariant subspace. An operator with both of these properties has a rank-one operator in its commutant.
\begin{lemma}
\label{InjectiveDenseRange}
Let $T\in \mathcal{L}(X)$ be an operator that is not injective and without a dense range. Then there exists a rank-one operator $F\in \mathcal{L}(X)$ such that $T \leftrightarrow F$. 
\end{lemma}
\begin{proof}
Take $0\neq y\in \ker(T)$ and $0\neq f\in X^*$ such that $\operatorname{Range}(T) \subseteq \ker(f)$. Consider a rank-one operator $f\otimes y\in \mathcal{L}(X)$ given by $(f\otimes y)(x): = f(x)y$. Then it is easy to check that $T \leftrightarrow f\otimes y$.
\end{proof}

The following result is an immediate consequence of Lemma \ref{InjectiveDenseRange} and provides a useful tool for constructing commuting chains. Throughout this paper, by a projection on a Banach space $X$ we mean a bounded idempotent operator $P\in\mathcal{L}(X)$. This convention also applies when $X$ is a Hilbert space; in particular, a projection is not assumed to be orthogonal.

\begin{cor}
\label{ProjectionandFiniteRank}
Let $T\in \mathcal{L}(X)$ be a nilpotent operator or a projection. Then there exists a rank-one operator $F\in \mathcal{L}(X)$ such that $T \leftrightarrow F$. 
\end{cor} 

A natural question to ask is whether just one assumption from Lemma \ref{InjectiveDenseRange} is sufficient. In general, an operator $T$ being non-injective or lacking a dense range does not imply the existence of a rank-one operator $F$ commuting with $T$. It turns out that in both cases, such an operator need not even be Lomonosov.

To see this, recall that, for $T\in\mathcal{L}(X)$, the range of $T$ is dense in $X$ if and only if its Banach adjoint $T'\in\mathcal{L}(X^*)$ is injective. When $T\in\mathcal{L}(H)$ acts on a Hilbert space $H$, we continue to use $T'\in\mathcal{L}(H^*)$ for its Banach adjoint, while $T^*\in\mathcal{L}(H)$ denotes its Hilbert space adjoint.

Moreover, we need the following elementary fact.
\begin{lemma}
\label{adjointChainPreserved}
Let $T\in \mathcal{L}(X)$. If $T$ is of chain $N$, then $T'$ is also of chain $N$. The converse holds if $X$ is reflexive.
\end{lemma}

\begin{example}
\label{NonInjectiveNonLomonosov}
We first provide an example of an operator without a dense range that fails to be Lomonosov. Because this is precisely the operator $T$ from Theorem \ref{Hadwin}, and since we will refer to it later in the paper, we outline its construction below. Comprehensive details can be found in the survey paper by A. Shields \cite[Section II]{Shields}. 

Let $\beta=(\beta_n)_{n=0}^\infty$ be a sequence of strictly positive numbers with $\beta_0=1$. We define $$H^2(\beta):=\{f(z) = \sum_{n=0}^\infty a_n z^n: (\beta_na_n)_{n=0}^\infty \in \ell_2\}.$$ 
While these sums are formal expressions in general, for certain weights $\beta$ they represent analytic functions on the open unit disk. We restrict our focus to such sequences (e.g., $\beta_n = e^{\sqrt{n}}$), as they were used in \cite{Hadwin}. The space $H^2(\beta)$ is a Hilbert space under the inner product $\langle f,g\rangle = \sum_{n=0}^\infty a_n\overline{b_n} |\beta_n|^2$, where $g(z)=\sum_{n\geq 0} b_nz^n$. Furthermore, the sequence $(e_n)_{n=0}^\infty$, defined by $e_n = \frac{z^n}{\beta_n}$, for $n\geq 0$, forms an orthonormal basis for $H^2(\beta)$. 

We define the multiplication operator $M_z : H^2(\beta) \rightarrow H^2(\beta)$ by
$$(M_zf)(z) := \sum_{n=0}^\infty a_n z^{n+1},$$
for any $f\in H^2(\beta)$, where $f(z)=\sum_{n\geq 0}a_n z^n$. Hadwin et al. \cite{Hadwin} demonstrated that under an additional restriction on $\beta$, the operator $M_z$ is non-Lomonosov. Moreover, we observe that $\operatorname{Range}(M_z)\subseteq \operatorname{span}\{e_n: n\geq 1\}$. As a consequence, $\overline{\operatorname{Range}(M_z)} \neq H^2(\beta)$, establishing that a non-dense range does not guarantee that an operator is Lomonosov.

By duality, we now consider the Banach adjoint $(M_z)'\in \mathcal{L}(H^2(\beta)^*) \cong \mathcal{L}(H^2(\beta))$. Because $M_z$ lacks a dense range, $(M_z)'$ must be non-injective. Furthermore, from Lemma \ref{adjointChainPreserved}, the fact that $M_z$ is non-Lomonosov implies that $(M_z)'$ is also non-Lomonosov. This shows that non-injectivity of an operator does not imply it being Lomonosov.
\end{example}

Next, we note that if either an operator $T$ or its Banach adjoint $T'$ has an eigenvalue, $T$ is not necessarily of chain $2$, demonstrated when considering the operators $\lambda I - M_z$ and $\lambda I - (M_z)'$, where $\lambda \in \C$ and $M_z$ is the operator from Example \ref{NonInjectiveNonLomonosov}. However, the following result shows that if both $T$ and $T'$ share a common eigenvalue, the conclusion is much stronger. This fact will be particularly useful in Section \ref{sec:reducton_to_rank-one}.

\begin{lemma}
\label{Eigenvalue}
Let $T\in \mathcal{L}(X)$ be such that both $T$ and $T'$ share an eigenvalue $\lambda \in \mathbb{F}$. Then there exists a rank-one operator $F\in \mathcal{L}(X)$ such that $T \leftrightarrow F$.
\end{lemma}

\begin{proof}
Let $\lambda \in \mathbb{F}$ be an eigenvalue for $T$ and $T'$. Consequently, the operators $\lambda I - T$ and $\lambda I - T'$ are non-injective. In particular, $\lambda I-T$ fails to have a dense range. By Lemma \ref{InjectiveDenseRange}, there exists a rank-one operator $F\in \mathcal{L}(X)$ which commutes with $\lambda I-T$, therefore with $T$.
\end{proof}

\begin{remark}
When $X$ is a complex Hilbert space, the relevant condition may equivalently be formulated by requiring that there exist an eigenvalue $\lambda$ of $T$ such that $\overline{\lambda}$ is an eigenvalue of its Hilbert space adjoint $T^*$.
\end{remark}

Our next goal is to construct longer commuting chains using the results from this section. We will use the class of \textit{reducing operators} to form such chains in Section \ref{sec:non-lomonosov_shifts}. Recall that for an operator $T\in \mathcal{L}(X)$, a pair $(V,W)$ of closed subspaces of $X$ is said to be a \textit{reducing pair} for $T$ if $X=V\oplus W$ and both $V$ and $W$ are $T$-invariant. An operator $T$ is then said to be \textit{reducing} if it admits a non-trivial proper reducing pair $(V,W)$.

The following is a classical fact (see, e.g., \cite[Theorem 2.22]{Invitation}).

\begin{theorem}
\label{ReducingCommutes}
An operator $T\in \mathcal{L}(X)$ is reducing if and only if there exists a projection $P\in \mathcal{N}(X)$ such that $T\leftrightarrow P$.  
\end{theorem}

It allows us to immediately deduce the following.

\begin{lemma}
\label{ReducingIsLomonosov}
Let $T\in \mathcal{L}(X)$ be a reducing operator. Then $T$ is of chain $2$.
\end{lemma}
\begin{proof}
Let $P\in \mathcal{N}(X)$ be a projection commuting with $T$. From Corollary \ref{ProjectionandFiniteRank} there exists a rank-one operator $F\in \mathcal{L}(X)$ such that $P \leftrightarrow F$. Thus we obtain a chain $T \leftrightarrow P \leftrightarrow F$. 
\end{proof}

The converse of Lemma \ref{ReducingIsLomonosov} does not hold in general. To see this, let $K: L_2[0,1] \rightarrow L_2[0,1]$ be the Volterra operator given by
\begin{equation}
\label{exampleVolterra}
(Kf)(t) := \int_0^t f(s)\,ds \quad \text{for } f\in L_2[0,1] \text{ and } t\in [0,1].
\end{equation}
It is known that $K$ is a compact operator, so in particular $K$ is of chain $2$. Furthermore, a closed subspace $V \subseteq L_2[0,1]$ is $K$-invariant if and only it is of the form
\begin{equation}
\label{exampleVolterra_subspaces}
V = \{f\in L_2[0,1]: f=0\text{ on } [0,a] \text{ a.e.}\},
\end{equation}
for some $0\leq a\leq 1$ (see, e.g., \cite[Theorem 7.4.1]{garcia2023operator}). In particular, it is not possible to find a non-trivial proper pair $(V,W)$ that reduces $K$.

The final component from this section that will be used in the proof of Theorem \ref{non-LomonosovChain3} is that chains of commuting operators are stable under similarity. Let $T_1\in \mathcal{L}(X)$ and $T_2\in \mathcal{L}(Y)$ be operators on Banach spaces $X$ and $Y$. We say that $T_1$ is \textit{similar} to $T_2$ if there exists an invertible operator $R\in \mathcal{L}(Y,X)$ such that $T_1=RT_2R^{-1}$. 

\begin{lemma}
\label{similarChainPreserved}
Let $X$ and $Y$ be Banach spaces and suppose the operators $T_1\in \mathcal{L}(X)$ and $T_2\in \mathcal{L}(Y)$ are similar. If $T_1$ is of chain $N$, then $T_2$ is also of chain $N$.
\end{lemma}
\begin{proof}
Note that if $S \in \mathcal{N}(X)$ then $R^{-1}SR \in \mathcal{N}(Y)$. Using this it is easy to verify that if $T_1$ has a chain
$$T_1 \leftrightarrow S_1 \leftrightarrow S_2 \leftrightarrow \dots\leftrightarrow S_{N-1} \leftrightarrow K,$$
then $T_2$ has a chain
$$T_2 \leftrightarrow R^{-1}S_1R \leftrightarrow R^{-1}S_2R \leftrightarrow \dots \leftrightarrow R^{-1}S_{N-1}R \leftrightarrow R^{-1}KR.$$
The claim also trivially holds when $T_1\in \mathcal{K}(X)$ or when $T_1 \leftrightarrow K$. 
\end{proof}

Using this notion, we can easily deduce that every normal operator on a separable Hilbert space $H$ is of chain $2$. Indeed, given a normal operator $T\in \mathcal{L}(H)$, it is known that $T$ is similar to a multiplication operator $M_f : L_2(\mu) \rightarrow L_2(\mu)$ for some finite measure $\mu$ and $f\in L_\infty(\mu)$ (see, e.g., \cite[Theorem 1.6]{InvariantSubspaces}). Furthermore, \cite[Theorem 2.8]{AbramovicMultiplication} establishes that $M_f$ is of chain $2$. By Lemma \ref{similarChainPreserved}, it follows that $T$ must also be of chain $2$.

The above also holds for non-separable Hilbert spaces. In Section \ref{sec:commuting_graphs} we state Theorem \ref{nonseparableHilbert} which allows us to conclude that every operator (and thus, in particular, any normal operator) on a non-separable Hilbert space is of chain $2$.

\begin{remark}
It is worth noting that multiplication operators defined on $C(K)$ spaces, where $K$ is any compact Hausdorff topological space, are also of chain $2$ (see \cite[Corollary 2.7]{AbramovicMultiplication}).
\end{remark}

To finish this section, we consider one more relation between operators. Let $T_1 \in \mathcal{L}(X)$ and $T_2 \in \mathcal{L}(Y)$ be operators on Banach spaces $X$ and $Y$. We say that $T_1$ is \textit{quasi-similar} to $T_2$ if there exist operators $A\in \mathcal{L}(Y,X)$ and $B\in \mathcal{L}(X,Y)$, that are injective and have dense ranges, such that $T_1 A=AT_2$ and $BT_1 = T_2 B.$

Note that similar operators are quasi-similar. By Lemma \ref{similarChainPreserved}, similarity preserves chains of any length. The following result shows that quasi-similarity preserves chains of length $1$. 
\begin{lemma}
\label{quasi-similar}
Let $X$ and $Y$ be Banach spaces and suppose $T_1\in \mathcal{L}(X)$ and $T_2\in \mathcal{L}(Y)$ are quasi-similar via operators $A\in \mathcal{L}(Y,X)$ and $B\in \mathcal{L}(X,Y)$. If $T_1$ is of chain $1$, then $T_2$ is also of chain $1$.
\end{lemma}
\begin{proof}
Suppose $T_1K_1 = K_1T_1$ for some non-zero operator $K_1\in \mathcal{K}(X)$. Consider $K_2 = BK_1A.$ Then $K_2 \in \mathcal{K}(Y)$. Moreover,
$$T_2K_2 = T_2 (BK_1A) = BT_1K_1A = BK_1T_1A = (BK_1A)T_2 = K_2 T_2.$$
It is left to show that $K_2 \neq 0$. Suppose otherwise. Then $BK_1A=0$, which implies that $\operatorname{Range}(A) \subseteq \ker (BK_1)$. Using the fact that $\overline{\operatorname{Range}(A)} =X$  we get $BK_1 = 0$, so $\operatorname{Range}(K_1) \subseteq \ker(B)$. Since $\ker(B)=\{0\}$, this implies $K_1=0$, giving a contradiction.
\end{proof}

\begin{remark}
The preceding proof also guarantees that if $T_1 \in \mathcal{K}(X)$ is quasi-similar to $T_2$, then $T_2$ is of chain $1$. Interestingly, in general, we cannot expect for $T_2$ to be compact. T. Hoover \cite{Hoover} provided an example of two quasi-similar operators $T_1$ and $T_2$, where $T_1$ is compact and $T_2$ is non-compact. In particular, this $T_2$ is of minimal chain $1$.
\end{remark}

\section{Non-Lomonosov shifts of chain 3}
\label{sec:non-lomonosov_shifts}

The primary objective of this section is to establish Theorem \ref{non-LomonosovChain3}. We will demonstrate that for a specific family of operators that includes those from Theorem \ref{non-LomonosovChain3}, there always exists a commuting chain
$$T \leftrightarrow S_1 \leftrightarrow S_2 \leftrightarrow K.$$
More precisely, we will show that on a Banach space with an unconditional basis $(e_n)_{n=1}^\infty$, every operator defined by $Te_n = w_n e_{\sigma(n)}$, where $(w_n)_{n=1}^\infty$ is a scalar sequence and $\sigma : \N \rightarrow \N$ is injective, is of chain $3$.

The following lemma provides a suitable set, which allows us to define a projection $P\in \mathcal{N}(X)$ for the purpose of constructing a commuting chain.

\begin{lemma}
\label{sigmaset}
Let $\sigma:\N \rightarrow \N$ be injective. There exists a proper, non-empty subset $B \subsetneq \N$ with the following property: for every $n\in \N$,
$$n \in B \iff \sigma^2(n)\in B.$$ 
\end{lemma}
\begin{proof}
We consider two separate cases.

Case 1. $\sigma$ is onto. Fix $n_0 \in \N$ and consider the sets
$$A = \{\sigma^k(n_0): k\in \Z\} ~\text{ and } ~B = \{\sigma^{2k}(n_0): k \in \Z\}.$$
First, note that $\varnothing \neq B \subseteq A \subseteq \N$. Our goal is to show that $B$ is a proper subset of $\N$. If $A$ is a proper subset of $\N$, then automatically so is $B$. Let us assume that $A=\N$. Suppose, by way of contradiction, that $B=\N=A$. This implies that $\sigma(n_0) \in B$, so there exists a $k\in \Z$ such that $\sigma(n_0) = \sigma^{2k}(n_0)$. But then
\begin{align*}
A &= \{\sigma^{-2k+1}(n_0),\sigma^{-2k+2}(n_0),\dots,\sigma^{-1}(n_0),n_0, \sigma(n_0),\dots,\sigma^{2k-1}(n_0)\}. 
\end{align*}
In particular $A$ is finite, yielding a contradiction. Hence $B\subsetneq\N$. Moreover, it is straightforward to check that 
$$n \in B \iff \sigma^2(n)\in B,$$
which concludes Case 1.

Case 2. $\sigma$ is not onto. Just like in Case 1, we fix $n_0 \in \N$ but we additionally assume that $n_0\notin \operatorname{Range} (\sigma)$. Consider the sets 
$$A = \{\sigma^k(n_0): k\in \N\cup\{0\}\} ~\text{ and } ~B = \{\sigma^{2k}(n_0): k \in \N\cup \{0\}\}.$$
We will show that $B$ satisfies the claim. To show $B\neq \N$, following the same argument as in Case 1, suppose, by way of contradiction, that $B=\N=A$. Once again, since $\sigma(n_0) \in B$, we can find $k\in \N\cup \{0\}$ such that $\sigma(n_0) = \sigma^{2k}(n_0)$. But then
\begin{align*}
A =\{n_0,\sigma(n_0),\dots,\sigma^{2k-1}(n_0)\},
\end{align*}
which again results in a contradiction. Hence $B\subsetneq\N$. 

Finally we show that
$$n \in B \iff \sigma^2(n)\in B.$$ 
First, suppose $n\in B$. We can find $k\in \N\cup\{0\}$ such that
$n=\sigma^{2k}(n_0)$. Thus, $\sigma^2(n)=\sigma^{2(k+1)}(n_0) \in B$. 

Suppose now $n$ is such that $\sigma^2(n) \in B$. There exists a $k \in \N\cup\{0\}$ with $\sigma^2(n)=\sigma^{2k}(n_0)$. Note that $k\neq 0$, as this would imply $n_0=\sigma^2(n)\in \operatorname{Range}(\sigma^2) \subseteq \operatorname{Range}( \sigma)$, which contradicts $n_0 \notin \operatorname{Range}(\sigma)$. Therefore, $k\geq 1$ and $n=\sigma^{2k-2}(n_0) \in B$ follows.
\end{proof}

We are ready to prove the main result of this section.
\begin{prop}
\label{Chain4Main}
Let $X$ be a Banach space with an unconditional basis $(e_n)_{n=1}^\infty$. Let $T\in \mathcal{L}(X)$ be an operator satisfying
$$Te_n = w_n e_{\sigma(n)} \qquad n\geq 1,$$
where $(w_n)_{n=1}^\infty$ is a sequence of scalars and $\sigma:\N\rightarrow \N$ is an injective function. Then $T^2$ is reducing. In particular, $T$ is of chain $3$. 
\end{prop}
\begin{proof}
By Lemma \ref{sigmaset}, there exists a proper, non-empty subset $B \subsetneq \N$ such that $n\in B \iff \sigma^2(n)\in B$. 
We define $P_B : X \rightarrow X$ via
$$
P_B e_n = 
     \begin{cases}
       0 &\text{if } n \notin B,\\
       e_n &\text{if } n\in B.\\
     \end{cases}
$$
Then $P_B$ is a projection on $X$. Since $\varnothing \subsetneq B \subsetneq \N$ we get that $P_B \in \mathcal{N}(X)$. We claim that $T^2 \leftrightarrow P_B$. First, observe that for any $n\in \N$,
$$T^2 e_n = T(w_n e_{\sigma(n)}) = w_nw_{\sigma(n)}e_{\sigma^2(n)}.$$
Now, fix $n \in \N$. If $n\in B$, then $\sigma^2(n)\in B$, which yields
\begin{align*}
T^2P_B e_n = w_nw_{\sigma(n)}e_{\sigma^2(n)} = P_B T^2 e_n.
\end{align*}
Conversely, if $n\notin B$, then $\sigma^2(n)\notin B$, giving
\begin{align*}
T^2P_B e_n = 0 = P_B T^2 e_n.
\end{align*}
Thus, $T^2 \leftrightarrow P_B$. By Theorem \ref{ReducingCommutes}, the operator $T^2$ is reducing. Lemma \ref{ReducingIsLomonosov} then implies that $T^2$ is of chain $2$. Finally, since $T \leftrightarrow T^2$, it follows that $T$ is of chain $3$
\end{proof}

Our next goal is to apply Proposition \ref{Chain4Main} to the class of operators from Theorem \ref{Hadwin}, focusing specifically on the operator $M_z\in \mathcal{L}(H^2(\beta))$ presented in Example \ref{NonInjectiveNonLomonosov}. To do so, we require the following crucial fact (see \cite[Section II, Proposition 7]{Shields}).

\begin{prop}
\label{UnitaryEquivalence}
Let $M_z \in \mathcal{L}(H^2(\beta))$. Then $M_z$ is unitary equivalent to a weighted shift $T\in \mathcal{L}(\ell_2)$ defined by
$$Te_n = w_n e_{n+1} \qquad n\geq 1,$$
for some sequence of scalars $(w_n)_{n=1}^\infty$. 
\end{prop}

\begin{proof}[Proof of Theorem \ref{non-LomonosovChain3}]
By Proposition \ref{UnitaryEquivalence}, $M_z$ is unitary equivalent (thus in particular similar) to a weighted shift $T\in \mathcal{L}(\ell_2)$. As a consequence of Lemma \ref{similarChainPreserved}, $M_z$ is of the same chain as $T$. From Proposition \ref{Chain4Main}, it follows that $T$ is of chain $3$, hence so is $M_z$. In particular, if $\beta$ is chosen in such a way that $M_z$ is a non-Lomonosov operator, we conclude that $3$ is the minimal chain of $M_z$.
\end{proof}

\begin{remark}
Theorem \ref{non-LomonosovChain3} can be also verified directly. One can show that $(M_z)^2 = M_{z^2}$ commutes with the projection $P\in \mathcal{N}(H^2(\beta))$ given by
$$
Pz^n = 
     \begin{cases}
       0 &\text{if $n$ is odd},\\
       z^n &\text{if $n$ is even}.\\
     \end{cases}
$$
By Corollary \ref{ProjectionandFiniteRank}, there exists a rank-one operator $F\in \mathcal{L}(H^2(\beta))$ such that $P\leftrightarrow F$. This can be also shown directly by considering
$$   
Fz^n = 
     \begin{cases}
       0 &\text{if } n\geq 1,\\
       z^0 &\text{if } n=0.\\
     \end{cases}
$$
Thus we explicitly construct the chain $M_z \leftrightarrow M_{z^2} \leftrightarrow P \leftrightarrow F.$
\end{remark}

\section{Reduction to rank-one operators}
\label{sec:reducton_to_rank-one}

So far in this article, when considering chains of commuting operators, the final operator in the chain has been not only a non-zero compact operator, but even a rank-one operator. It is interesting to ask whether this is always the case. More precisely, suppose an operator $T\in \mathcal{L}(X)$ is of chain $N_1$ for some $N_1\geq 1$. Is it possible to construct a new chain of length $N_2\geq N_1$ starting from $T$ that allows us to connect $T$ to a rank-one operator? We start with an observation that, in general, we should expect $N_2 > N_1$.

\begin{example}
\label{Volterra2}
Let $K:L_2[0,1]\rightarrow L_2[0,1]$ be the Volterra operator defined by equation \eqref{exampleVolterra}. We first observe that $K$ does not commute with any non-zero finite-rank operator $F$. If such an $F$ were to exist, its range $\operatorname{Range}(F)$ would be $K$-invariant. This gives a contradiction, since formula \eqref{exampleVolterra_subspaces} demonstrates that $K$ has no finite-dimensional invariant subspaces. Consequently, the chain $K \leftrightarrow F$ cannot hold for any non-zero finite-rank operator $F$.

Now, consider the operator $T = I+K$. Because $\{T\}'=\{K\}'$, it follows that $T$ also fails to commute with any finite-rank operator. Furthermore, while $T$ itself is not compact, it commutes with the compact operator $K$. Therefore, $T$ is of chain $1$, despite failing to commute with any operator of finite rank.
\end{example}

This shows that, in general, if we want to connect an operator to a rank-one operator via a chain of commuting operators, then this chain must be longer. The natural next question is to determine exactly how much longer this chain must be. Equivalently, given a non-zero operator $K\in \mathcal{K}(X)$, what length of a chain is required to connect $K$ to a rank-one operator $F\in \mathcal{L}(X)$? The following lemma shows that for non-quasinilpotent compact operators on a complex Banach space, one additional operator suffices. Recall that an operator $S\in \mathcal{L}(X)$ is called \textit{quasinilpotent} if its spectrum satisfies $\sigma(S)=\{0\}$. We also denote by $\sigma_p(S)$ the point spectrum of $S$.

\begin{lemma}
\label{CompactFiniteRank}
Let $X$ be a complex Banach space and let $K\in \mathcal{K}(X)$ be a non-quasinilpotent operator. Then there exists a rank-one operator $F\in \mathcal{L}(X)$ such that $K \leftrightarrow F.$
\end{lemma}
\begin{proof}
Fix $0\neq \lambda \in \sigma(K)$. As $K$ is compact, we get $\lambda \in \sigma_p(K)$. Since $\sigma(K)=\sigma(K')$, we also have $\lambda \in \sigma(K')$. Moreover, compactness of $K'$ implies $\lambda \in \sigma_p(K')$. From Lemma \ref{Eigenvalue}, there exists a rank-one operator $F\in \mathcal{L}(X)$ such that $K \leftrightarrow F$.
\end{proof}

In the case when $X$ is a real Banach space, commutant of $K$ may contain a rank-two operator instead.

\begin{prop}
\label{CompactFiniteRank2}
Let $X$ be a real Banach space and let $K\in \mathcal{K}(X)$ be a non-quasinilpotent operator. Then there exists a rank-one or rank-two operator $F\in \mathcal{L}(X)$ such that $K\leftrightarrow F$. 
\end{prop}

\begin{proof}
Let $K\in \mathcal{K}(X)$ be non-quasinilpotent. If $\sigma_p(K)$ contains a real eigenvalue $\lambda$, then the same argument used in Lemma \ref{CompactFiniteRank} guarantees the existence of a (real) rank-one operator $F$ that commutes with $K$.

Now, suppose $\sigma_p(K)$ consists exclusively of non-real eigenvalues. We pass to the complexi\-fication $X_\C$ and consider the complexified operator $K_\C \in \mathcal{K}(X_\C)$. Fix $0\neq \lambda \in \sigma_p(K_\C)$. It follows that $\lambda \in \sigma_p(K_\C')$ as well. As demonstrated in the proof of Lemma \ref{Eigenvalue}, $K_\C$ commutes with the rank-one operator $f\otimes x\in \mathcal{L}(X_\C)$, where $0\neq f\in \ker(\lambda I - K_\C')$ and $0\neq x\in \ker(\lambda I - K_\C)$. By taking complex conjugates, we obtain $0\neq \overline{x}\in \ker(\overline{\lambda} I - K_\C)$ and $0\neq \overline{f} \in \ker(\overline{\lambda} I - K_\C')$, implying that $K_\C$ also commutes with $\overline{f}\otimes \overline{x}$. Finally, consider the rank-two operator $F_2 = f\otimes x + \overline{f}\otimes \overline{x}$. It is straightforward to verify that $F_2$ acts as a real operator on $X$. Furthermore, because $F_2$ commutes with $K_\C$, the same is true for $K$.
\end{proof}

Note that in any infinite-dimensional Banach space $X$, Lemma \ref{InjectiveDenseRange} ensures that every rank-two operator $F_2\in \mathcal{L}(X)$ commutes with a rank-one operator $F_1\in \mathcal{L}(X)$. Therefore, for any compact, non-quasinilpotent operator $K$ on a real Banach space, we always construct a chain $K \leftrightarrow F_2 \leftrightarrow F_1$. In the complex case, a direct chain $K \leftrightarrow F_1$ is always attainable.

It is a natural question to ask whether a similar chain exists for quasinilpotent compact operators. Since the Volterra operator is an example of such an operator, by Example \ref{Volterra2}, we can immediately deduce that a direct chain $K \leftrightarrow F$ to a rank-one operator $F$ need not exist in general. Nevertheless, the following example demonstrates that the Volterra operator can be connected to a rank-one operator $F$, via one additional nilpotent operator.

\begin{example}
Let $K : L_2[0,1] \rightarrow L_2[0,1]$ be the Volterra operator. It is useful to note that $K$ can be expressed as $Kf = \mathbbm{1} \ast f$, where $\ast$ denotes convolution. It is well known that $K$ is a quasinilpotent operator. As observed in Example \ref{Volterra2}, a direct chain $K \leftrightarrow F$ is not possible for any finite-rank $F$. Nevertheless, following the remark under \cite[Lemma 2.6]{Ambrozie}, we define the operator $M : L_2[0,1] \rightarrow L_2[0,1]$ by
$$(Mf)(t) := (\mathbbm{1}_{[\frac{1}{2},1]} \ast f)(t) = \int_0^1 \mathbbm{1}_{[\frac{1}{2},1]}(t)f(t-s) \,ds.$$
By the commutativity and associativity of convolution, it follows that $K \leftrightarrow M$. Furthermore, because $\mathbbm{1}_{[\frac{1}{2},1]} \ast \mathbbm{1}_{[\frac{1}{2},1]} = 0$, the operator $M$ is nilpotent of order two (i.e. $M^2=0$). Corollary \ref{ProjectionandFiniteRank} then guarantees the existence of a rank-one operator $F$ satisfying $M \leftrightarrow F$. Thus, we have explicitly constructed the chain
$$K \leftrightarrow M \leftrightarrow F.$$
\end{example}

\begin{problem}
Let $X$ be a Banach space and let $K\in \mathcal{K}(X)$ be a non-zero, quasinilpotent operator. Does there exist an operator $S\in \mathcal{N}(X)$ and a rank-one operator $F\in \mathcal{L}(X)$ such that 
$$K\leftrightarrow S \leftrightarrow F$$
holds?
\end{problem}

Finally, observe that the proofs of Lemma \ref{CompactFiniteRank} and Proposition \ref{CompactFiniteRank2} rely on spectral properties of compact operators. Consequently, the same statements will hold for the more general family of non-quasinilpotent \textit{strictly singular} operators (i.e. operators not bounded below on any infinite-dimensional closed subspace of $X$), as they share the same spectral properties as compact operators. However, in this case, we must additionally assume that the Banach adjoint of a given strictly singular operator is also strictly singular. While not true in general, it is known to hold for operators defined on $\ell_p$ for $1\leq p<\infty$, $L_p[0,1]$ for $1\leq p\leq\infty$ or $C[0,1]$. Interestingly, in general, for quasinilpotent strictly singular operators, we cannot guarantee a chain $T \leftrightarrow S \leftrightarrow F$, as there exist examples of such operators without invariant subspaces (see \cite{Read1999StrictlySO}).

\section{The theory of commuting graphs}
\label{sec:commuting_graphs}

Our goal in this section is to prove Theorem \ref{NoCompactAmbrozie} by showing that an operator constructed in \cite{Ambrozie} is precisely of this type. Because this operator was originally introduced within the context of commuting graphs, we first provide a brief overview of that theory.

We begin with some preliminary definitions. As usual, let $X$ denote a Banach space. We define the commuting graph $\Gamma(\mathcal{N}(X))$ as the graph whose vertices are the elements of $\mathcal{N}(X)$, with an edge between $T_1,T_2\in \mathcal{N}(X)$ if and only if they commute (i.e., $T_1 \leftrightarrow T_2$). The distance between $T_1$ and $T_2$ in $\Gamma(\mathcal{N}(X))$ is defined by
$$d(T_1,T_2) := \min\{N\in\N: \exists_{S_1,\dots,S_{N-1} \in \mathcal{N}(X)}~~T_1\leftrightarrow S_1\leftrightarrow S_2\leftrightarrow \dots\leftrightarrow S_{N-1}\leftrightarrow T_2\}.$$
If no such path exists, we set $d(T_1,T_2)=\infty$. If the distance between any two operators in $\mathcal{N}(X)$ is finite, we say that $\Gamma(\mathcal{N}(X))$ is \textit{connected}, and we define its diameter as 
$$\text{diam}(\Gamma(\mathcal{N}(X))) := \sup\{d(T_1,T_2) : T_1,T_2\in \mathcal{N}(X)\}.$$ 
If the graph is not connected, we say it is \textit{disconnected}.

Note the immediate similarity to the concept of commuting chains from this article. In this context, given an operator $T_1\in \mathcal{L}(X)$, one seeks for a chain connecting it to every operator $T_2\in \mathcal{N}(X)$, rather than to some non-zero compact operator. In particular, if $\text{diam}(\Gamma(\mathcal{N}(X)))=N$, then every operator $T\in \mathcal{L}(X)$ is of chain $N$. 

This topic has been studied extensively (see, e.g., \cite{Ambrozie}, \cite{Kuzma2}, \cite{Kuzma3} and \cite{Radjavi}). The goal is to determine for which spaces $X$ the graph $\Gamma(\mathcal{N}(X))$ is connected. Furthermore, whenever this is the case, one seeks to calculate its diameter. If $\Gamma(\mathcal{N}(X))$ is disconnected, the objective is to characterize the connected components of $\Gamma(\mathcal{N}(X))$. 

We now recall some known results in this area that are relevant to our theory of chains of commuting operators. First, the graph of commuting operators is connected on every finite dimensional Hilbert space of dimension greater than $2$. More precisely, if $H$ is a Hilbert space with $2<\dim H<\infty$, then $\operatorname{diam}(\Gamma(\mathcal{N}(H)))=4.$ The complex case was established in \cite[Corollary 7]{Radjavi}, while the real case was subsequently resolved in \cite{Miguel2013,Grau2017}.

The case where $\dim H=2$ is notably different. We present the proof of the following elementary result, as it provides valuable intuition for the infinite-dimensional separable Hilbert spaces treated later in this section.
\begin{lemma}
Let $H$ be a Hilbert space with $\dim H=2$. Then $\Gamma(\mathcal{N}(H))$ is disconnected.
\end{lemma}
\begin{proof}
Since $\dim H=2$, we identify $\mathcal{N}(H)$ with the space of non-scalar matrices in $M_2(\mathbb{F})$. Let 
\begin{align*}
T_1=\begin{bmatrix}
1 & 0\\
0 & 0 \\
\end{bmatrix}.
\end{align*}
From a direct calculation it follows that for $B_1 \in \mathcal{N}(H)$, $T_1\leftrightarrow B_1$ if and only if $B_1$ is of the form $B_1 = aI+bT_1$, for some $a,b\in \mathbb{F}, b\neq 0$. In particular, this implies that $\{B_1\}' = \{aI+bT_1\}' = \{T_1\}'$. This shows the commutant of $T_1$ stabilizes, hence $\Gamma(\mathcal{N}(H))$ is disconnected.
\end{proof}

Observe that in the preceding proof, we constructed an operator $T$ with the property that $\{T\}' \neq \mathcal{N}(H)$, and such that for every non-scalar operator $S\in \{T\}'$, we have $\{S\}'=\{T\}'$. For finite dimensional Hilbert spaces, such a construction is only possible when $\dim H=2$. 

The case of infinitely dimensional Hilbert space $H$ has also been investigated. 
\begin{theorem}[Corollary 2.2 in \cite{Ambrozie}] 
\label{nonseparableHilbert} Let $H$ be a non-separable Hilbert space. Then $\text{diam}(\Gamma(\mathcal{N}(H))) =2$
\end{theorem}

As an immediate corollary of Theorem \ref{nonseparableHilbert}, it follows that every operator $T$ on a non-separable Hilbert space is of chain $2$. It is natural to ask whether this result also holds for non-separable Banach spaces.

Next, we restrict our attention to infinite-dimensional separable Hilbert spaces. Interestingly, in this case, $\Gamma(\mathcal{N}(H))$ turns out to be disconnected, just as in dimension $2$.
\begin{theorem}[Theorem 2.3 in \cite{Ambrozie}]
\label{AmbrozieOperator}
Let $H$ be a separable infinite-dimensional Hilbert space over $\C$. Then there exists $T\in \mathcal{N}(H)$ with the following property: $\{T\}' \neq \mathcal{N}(H)$ and 
$$\text{ for every } S\in \mathcal{N}(H), \quad T\leftrightarrow S \implies \{S\}'=\{T\}'.$$
\end{theorem}

In particular, taking any $T_1 \in \mathcal{N}(H)\setminus\{T\}'$ guarantees that $d(T,T_1)=\infty$. Nevertheless, it is not immediately obvious whether this construction provides an operator that cannot be connected to a non-zero compact operator. One must verify that the commutant $\{T\}'$ contains no non-zero compact operator $K$. This indeed holds true, as demonstrated by the following result.

\begin{prop}
\label{AmbrozieOperatorNoCompact}
Let $T\in \mathcal{L}(H)$ be the operator from Theorem \ref{AmbrozieOperator}. Then the only compact operator $K \in \{T\}'$ is $K=0$.
\end{prop}

Proposition \ref{AmbrozieOperatorNoCompact}, combined with the stabilization of the commutant of $T$ described in Theorem \ref{AmbrozieOperator}, immediately proves Theorem \ref{NoCompactAmbrozie}.

The remainder of this section is devoted to proving Proposition \ref{AmbrozieOperatorNoCompact}. We begin by reviewing the definition of $T$. Because its construction is highly involved, we state only the facts necessary to prove Proposition \ref{AmbrozieOperatorNoCompact}. Full details of this construction can be found in Section $3$ of \cite{Ambrozie}. 

Let $H=\ell_2$. We fix an orthonormal basis $(e_k)_{k=0}^\infty$ of $\ell_2$. For this construction it is convenient to index the basis starting from zero. We first construct a bounded linear operator $T : c_{00} \rightarrow c_{00}$ on the space of finitely supported sequences $c_{00}$, which we will subsequently extend to a bounded linear operator $T: \ell_2 \rightarrow \ell_2$ with the desired property. 

We fix an increasing sequence $(r_k)_{k=0}^\infty \subseteq \N \cup \{0\}$ and a function $h: \N \rightarrow \N\cup\{0\}$ satisfying the following conditions:
\begin{enumerate}
\item[(i)] $r_0=0, r_1=4,~ 4r_k<r_{k+1}<6r_k$, for any $k\in \N$;
\item[(ii)] $h(k) \leq k-1$, for all $k\in \N$;
\item[(iii)] for all $j,n \in \N$ and each $s\in \{0,1,\dots,n-1\}$, there exist infinitely many $k\in \N$ that simultaneously satisfy $h(k)=j$ and $r_k \equiv s\mod n.$
\end{enumerate} 
The existence of such a sequence and function was established in \cite[Lemma 3.1]{Ambrozie}. We highlight an important feature of the sequence $(r_k)_{k=0}^\infty$, namely that $\lim_{k\rightarrow \infty}(r_k - r_{k-1} -1)=\infty$, which we will use later.  We also fix a decreasing sequence of real numbers $(\varepsilon_k)_{k=1}^\infty$ such that $0<\varepsilon_k<\frac{1}{2}$ for all $k\in\N$. 

Next, we construct an auxiliary sequence $(u_n)_{n=-\infty}^\infty \subseteq c_{00}$. We define $u_n = 0$ for $n<0$ and $u_0=e_0$. The remaining elements $u_n$ (for $n\geq 1)$ are constructed alongside the operator $T:c_{00} \rightarrow c_{00}$ so that its action on the basis elements is given by:
\begin{align*}
Te_0 &= 0, \quad Te_j = e_{j-1} \quad \text{if } r_k < j < r_{k+1},\\
Te_{r_k} &= \varepsilon_k e_{r_k-1} + \frac{\sqrt{\varepsilon_k}}{\|u_{h(k)}\|}u_{h(k)},\\
u_{r_k} &= \frac{1}{\varepsilon_1\dots \varepsilon_k}e_{r_k}, \quad\text{and}\quad
u_j = T^{r_k-j} u_{r_k}\quad \text{if } r_{k-1}<j<r_k.
\end{align*}
We then extend $T$ linearly to $c_{00}$. We also define
$$   
\omega_j =
     \begin{cases}
       1 &\text{if } j \neq r_k ,\text{ for all } k\in\N,\\
       \varepsilon_k &\text{if } j=r_k~,\text{ for some } k\in\N.\\
     \end{cases}
$$
It was shown (check the remark above \cite[Lemma 3.3]{Ambrozie}) that $\|T\|\leq 2$, provided that the sequence $(\varepsilon_k)_{k=1}^\infty$ decreases to zero sufficiently fast. This allows to extend $T$ to a bounded linear operator on $\ell_2$. With a slight abuse of notation, we also denote this extension by $T$. 

Next, we list crucial properties of the operator $T$ that allow us to prove Proposition \ref{AmbrozieOperatorNoCompact}.
\begin{lemma}[Lemma 3.4 and Lemma 3.3 (i) in \cite{Ambrozie}]
\label{commutantofAmbrozie}
Let $A\in \{T\}'.$ Then there exists a sequence $(c_i)_{i=0}^\infty\subseteq \C$ with $\sum_{i=0}^\infty |c_i|^2 <\infty$ such that 
\begin{align}
\label{commutantEquation}
Ax = \sum_{i=0}^\infty c_i T^i x \quad \text{for all } x\in c_{00}.
\end{align}
In particular, for every $j\in \N$, we have
$$Ae_j = c_0e_j+ \Big(\sum_{i=1}^j c_i\omega_j\dots \omega_{j-i+1}e_{j-i}\Big) + v \quad \text{for some } v\in \operatorname{span}\{e_0,\dots,e_{k-1}\},$$
where $k$ is chosen such that $r_k\leq j< r_{k+1}.$
\end{lemma}

We note that the right hand side of \eqref{commutantEquation} is well-defined, since for any finitely supported $x\in c_{00}$ we have $T^ix=0$ for all but finitely many $i$. We are ready to prove Proposition \ref{AmbrozieOperatorNoCompact}.

\begin{proof}[Proof of Proposition \ref{AmbrozieOperatorNoCompact}]
Fix a non-zero operator $A\in \{T\}'$. By Lemma \ref{commutantofAmbrozie}, there exists a sequence $(c_i)_{i=0}^\infty$ with $\sum_{i=0}^\infty |c_i|^2 <\infty$ such that $Ax=\sum_{i=0}^\infty c_i T^i x$ for all $x\in c_{00}$. Because $A\neq 0$, there must exist an index $j_0 \in \N$ for which $c_{j_0} \neq 0$, as otherwise, $A$ would vanish on the dense subspace $c_{00}$ and consequently on all of $\ell_2$. Next, we choose a sufficiently large integer $k_0\in\N$ so that $j_0 \leq r_{k_0}- r_{k_0-1}-1$. The existence of such a $k_0$ is guaranteed by the fact that $\lim_{k \rightarrow\infty}(r_k - r_{k-1}-1)=\infty$. Consider the sequence $(e_{r_k-1})_{k=k_0}^\infty$. We will show that $(Ae_{r_k-1})_{k=k_0}^\infty$ fails to have a convergent subsequence, implying that $A$ is not a compact operator. 

First observe that for every $k\geq k_0$ we have
\begin{align}
\label{FormulaAmbrozie}
Ae_{r_k-1} = c_0e_{r_k-1} + c_1e_{r_k-2}+c_2e_{r_k-3}+\dots+c_{r_k - r_{k-1}-1}e_{r_{k-1}} + u,
\end{align}
for some $u\in \operatorname{span}\{e_0,\dots,e_{r_{k-1}-1}\}$. To see this, fix $k\geq k_0$. Recall from the definition of the operator $T$ that for each index $j$ satisfying $r_k <j <r_{k+1}$ we have $Te_j = e_{j-1}$. In our case, since $r_{k-1}<r_k-1<r_k$, we obtain 
\begin{align}
\label{formula4}
T^ie_{r_k-1} = e_{r_k-i-1}\quad\text{for all } i=0,1,\dots,r_k-r_{k-1}-1.
\end{align}
Next, the latter part of Lemma \ref{commutantofAmbrozie} allows us to expand $Ae_{r_k-1}$ into the following finite sum:
\begin{align*}
Ae_{r_k-1} &= c_0e_{r_k-1} + c_1\omega_{r_k-1}e_{r_k-2}+c_2\omega_{r_k-1}\omega_{r_k-2}e_{r_k-3}+\dots\\
&+c_{r_k - r_{k-1}-1}\omega_{r_k-1}\omega_{r_k-2}\dots \omega_{r_{k-1}+1} e_{r_{k-1}}\\
&+c_{r_k - r_{k-1}-2}\omega_{r_k-1}\omega_{r_k-2}\dots \omega_{r_{k-1}+1}\omega_{r_{k-1}} e_{r_{k-1}-1}+\dots\\
&+c_{r_k-1}\omega_{r_k-1}\omega_{r_k-2}\dots \omega_1e_0 + v,
\end{align*}
for some $v\in \operatorname{span}\{e_0,\dots,e_{k-2}\}$. From the definition of $T$ and the first part of Lemma \ref{commutantofAmbrozie}, each $\omega_i$ corresponds to the weight appearing in the expansion of $T^ie_{r_k-1}$. Hence, using \eqref{formula4}, we deduce that $\omega_{r_k-1}=\dots=\omega_{r_{k-1}+1}=1$. This simplifies the expression to:
\begin{align*}
Ae_{r_k-1} &= c_0e_{r_k-1} + c_1e_{r_k-2}+c_2e_{r_k-3}+\dots+c_{r_k - r_{k-1}-1}e_{r_{k-1}}+\dots\\
&+c_{r_k - r_{k-1}-2}\omega_{r_{k-1}} e_{r_{k-1}-1}+\dots+c_{r_k-1}\omega_{r_{k-1}}\dots \omega_1 e_0 + v.
\end{align*}
Finally, letting
$$u:=c_{r_k - r_{k-1}-2}\omega_{r_{k-1}} e_{r_{k-1}-1}+\dots+c_{r_k-1}\omega_{r_{k-1}}\dots \omega_1 e_0 + v,$$
we see that $u\in \operatorname{span}\{e_0,\dots,e_{r_{k-1}-1}\}$, which proves \eqref{FormulaAmbrozie}.

Now, fix $n,k\in \N$ with $k_0\leq n<k$. Using \eqref{FormulaAmbrozie} we can write 
\begin{align*}
Ae_{r_k-1} - Ae_{r_n-1} &= c_0e_{r_k-1} + c_1e_{r_k-2}+c_2e_{r_k-3}+\dots+c_{r_k - r_{k-1}-1}e_{r_{k-1}}+u - Ae_{r_n-1}
\end{align*}
for some $u\in \operatorname{span}\{e_0,\dots,e_{r_{k-1}-1}\}$. Since $Ae_{r_n-1} \in \operatorname{span} \{e_0,\dots,e_{r_n-1}\}$ and $k>n$, it follows that 
$$u-Ae_{r_n-1}\in \operatorname{span}\{e_0,\dots,e_{r_{k-1}-1}\}.$$ 
In particular, the vector $z:=c_0e_{r_k-1} + c_1e_{r_k-2}+c_2e_{r_k-3}+\dots+c_{r_k - r_{k-1}-1}e_{r_{k-1}}$ is orthogonal to $u - Ae_{r_n-1}$. Thus,
\begin{align*}
\|Ae_{r_k-1} - Ae_{r_n-1}\|^2 &=\|z\|^2 +\|u-Ae_{r_n-1} \|^2 \geq \|z\|^2 = \sum_{i=0}^{r_k - r_{k-1}-1} |c_i|^2.
\end{align*}
Finally, due to the choice of $k_0$ and the property that $4r_{k-1}<r_k<6r_{k-1}$ for the sequence $(r_k)_{k=0}^\infty$, we obtain 
$$\sum_{i=0}^{r_k - r_{k-1}-1} |c_i|^2 \geq \sum_{i=0}^{r_{k_0} - r_{k_0-1}-1} |c_i|^2\geq |c_{j_0}|^2.$$ 
This implies that   
$$\|Ae_{r_k-1} - Ae_{r_n-1}\|^2 \geq |c_{j_0}|^2 >0,$$
for all $k_0\leq n<k$. Consequently, no subsequence of $(Ae_{r_k-1})_{k=k_0}^\infty$ is Cauchy, and thus none can converge. Hence, $A$ is not a compact operator.

\end{proof}

Even though not every operator can be connected to a compact operator via a chain of commutation, it is interesting to investigate for which operators this is true.
\begin{problem}
Which classes of operators on $\ell_2$ can be connected to a non-zero compact operator via a finite chain?
\end{problem}

Next, all the examples of operators in this article that can be connected to a non-zero compact operator are of chain $3$. A natural question to ask is the following:
\begin{problem}
Given $N\geq 4$, does there exist $T\in \mathcal{L}(X)$ that is of minimal chain $N$?
\end{problem}

As a final remark, we note the following intriguing observation. Recall that an operator $T\in \mathcal{L}(X)$ is said to have a \textit{hyperinvariant subspace}, if it is an invariant subspace for every operator $S\in \{T\}'$. Currently, there are no known examples of operators on complex $\ell_2$ that have an invariant subspace but fail to have a hyperinvariant subspace. It remains an open question whether every operator on $\ell_2$ with an invariant subspace necessarily admits a hyperinvariant subspace. If this were true, any operator $T\in \mathcal{L}(\ell_2)$ of chain $N$ (for any $N\geq 0$) would be guaranteed to have a hyperinvariant subspace. In particular, if every operator on $\ell_2$ were of chain $N$ for some $N\geq 0$, the Invariant Subspace Problem on $\ell_2$ would be solved. However, by Theorem \ref{NoCompactAmbrozie}, we conclude that such a universal chain property cannot hold.

I am grateful to Vladimir Troitsky and Adi Tcaciuc for valuable conversations on the topic of this paper. Moreover, I am thankful to Laurent Marcoux for directing me to the work of \cite{Ambrozie}, as well as to the organizers of the Banach Algebra and Operator Algebra 2024 conference at the University of Waterloo, that provided an opportunity for this discussion to happen. I would also like to thank the anonymous referees for their helpful suggestions, which improved the presentation and readability of the paper.
\bibliographystyle{alpha}
\bibliography{bib}

@article{Lomonosov73,
    author = {Lomonosov, V. J.},
    title  = {Invariant subspaces of the family of operators that commute with a completely
continuous operator},
    journal = {Funkcional. Anal. i Prilozen},
    volume = {7},
  	pages = {55-56},
    year   = {1973}
}

@article{Enflo,
  title={On the invariant subspace problem for {B}anach spaces},
  author={Enflo, P.},
  journal={Acta Mathematica},
  year={1987},
  volume={158},
  pages={213-313},
  url={https://api.semanticscholar.org/CorpusID:120558318}
}

@article{Read1986,
  title={A Short Proof Concerning the Invariant Subspace Problem},
  author={Read, C. J.},
  journal={Journal of The London Mathematical Society-second Series},
  year={1986},
  pages={335-348},
  url={https://api.semanticscholar.org/CorpusID:121764552}
}

@article{Read1985,
  title={A Solution to the Invariant Subspace Problem on the Space $\ell_1$},
  author={Read, C. J.},
  journal={Bulletin of The London Mathematical Society},
  year={1985},
  volume={17},
  pages={305-317},
  url={https://api.semanticscholar.org/CorpusID:120880059}
}

@article{Hadwin,
  title={An operator not satisfying {L}omonosov's hypothesis},
  author={D. W. Hadwin and E. A. Nordgren and H. Radjavi and P. Rosenthal},
  journal={Journal of Functional Analysis},
  year={1980},
  volume={38},
  pages={410-415},
  url={https://api.semanticscholar.org/CorpusID:121624855}
}

@ARTICLE{Troitsky,
  author = {Troitsky, V. G.},  
  title = {Lomonosov's Theorem Cannot Be Extended to Chains of Four Operators.},
  journal = {Proceedings of the American Mathematical Society}, 
  year={2000},
  volume={128},
  pages={521-525},
}

@article{AbramovicMultiplication,
  title={Multiplication and compact-friendly operators},
  author={Abramovich, Y. A. and Aliprantis, C. D. and Burkinshaw, O.},
  journal={Positivity},
  volume={1},
  pages={171--180},
  year={1997},
  publisher={Springer}
}

@book{Shields,
  title={{Topics in Operator Theory}},
  author={Pearcy, C. M.},
  isbn={9780821815137},
  lccn={74008254},
  series={Mathematical Surveys and Monographs},
  url={https://books.google.ca/books?id=bmSFAwAAQBAJ},
  year={1974},
  publisher={American Mathematical Society}
}

@book{InvariantSubspaces,
  title={Invariant Subspaces},
  author={Radjavi, H. and Rosenthal, P.},
  isbn={9780486428222},
  lccn={2003043828},
  series={Dover books on mathematics},
  url={https://books.google.ca/books?id=68SRnYbWKXQC},
  year={2003},
  publisher={Dover Publications}
}

@book{garcia2023operator,
  title={{Operator Theory by Example}},
  author={Garcia, S. R. and Mashreghi, J. and Ross, W. T.},
  isbn={9780192678850},
  series={Oxford Graduate Texts in Mathematics},
  url={https://books.google.ca/books?id=Rf-pEAAAQBAJ},
  year={2023},
  publisher={OUP Oxford}
}

@article{Radjavi,
title = {On the diameters of commuting graphs},
journal = {Linear Algebra and its Applications},
volume = {418},
number = {1},
pages = {161-176},
year = {2006},
issn = {0024-3795},
doi = {https://doi.org/10.1016/j.laa.2006.01.029},
url = {https://www.sciencedirect.com/science/article/pii/S0024379506000590},
author = {Akbari, S. and Mohammadian, A. M. and Radjavi, H. and Raja, P.},
}

@article{Ambrozie,
title = {The commuting graph of bounded linear operators on a {H}ilbert space},
journal = {Journal of Functional Analysis},
volume = {264},
number = {4},
pages = {1068-1087},
year = {2013},
issn = {0022-1236},
doi = {https://doi.org/10.1016/j.jfa.2012.11.011},
url = {https://www.sciencedirect.com/science/article/pii/S002212361200420X},
author = {Ambrozie, C. and Bračič, J. and Kuzma, B. and Müller, V.},
}

@book{Invitation,
  title={{An Invitation to Operator Theory}},
  author={Abramovich, Y. A. and Aliprantis, C. D.},
  isbn={9780821821466},
  lccn={20274420},
  year={2002},
  publisher={American Mathematical Society}
}

@article {Kuzma2,
    AUTHOR = {Dolinar, G. and Kuzma, B. and Oblak, P.},
     TITLE = {On maximal distances in a commuting graph},
   JOURNAL = {Electron. J. Linear Algebra},
  FJOURNAL = {Electronic Journal of Linear Algebra},
    VOLUME = {23},
      YEAR = {2012},
     PAGES = {243--256},
      ISSN = {1081-3810},
   MRCLASS = {15A27 (05C12 05C50)},
  MRNUMBER = {2889585},
MRREVIEWER = {Jan\ Foniok},
       DOI = {10.13001/1081-3810.1518},
       URL = {https://doi.org/10.13001/1081-3810.1518},
}

@article {Kuzma3,
title = {On diameter of the commuting graph of a full matrix algebra over a finite field},
journal = {Finite Fields and Their Applications},
volume = {37},
pages = {36-45},
year = {2016},
issn = {1071-5797},
doi = {https://doi.org/10.1016/j.ffa.2015.08.002},
url = {https://www.sciencedirect.com/science/article/pii/S1071579715000799},
author = {D. Dolžan and D. {Kokol Bukovšek} and B. Kuzma and P. Oblak},
}

@article{Hoover,
  title={Quasi-similarity of operators},
  author={Hoover, T. B.},
  journal={Illinois Journal of Mathematics},
  year={1972},
  volume={16},
  pages={678-686},
  url={https://api.semanticscholar.org/CorpusID:118045103}
}

@article{Read1999StrictlySO,
  title={Strictly singular operators and the invariant subspace problem},
  author={Read, C. J.},
  journal={Studia Mathematica},
  year={1999},
  volume={132},
  pages={203-226},
  url={https://api.semanticscholar.org/CorpusID:116739186}
}

@article{Miguel2013,
author={Miguel, C.},
title={A note on a conjecture about commuting graphs},
journal={Linear algebra and its applications},
year={2013},
publisher={Elsevier},
volume={438},
number={12},
pages={4750-4756},
issn={1873-1856},
doi={10.1016/j.laa.2013.02.016},
url={https://doi.org/10.1016/j.laa.2013.02.016}
}

@article {Grau2017,
    AUTHOR = {Grau, J. M. and Oller-Marc\'en, A. M. and Tasis, C.},
     TITLE = {On the diameter of the commuting graph of the full matrix ring
              over the real numbers},
   JOURNAL = {Bull. Iranian Math. Soc.},
  FJOURNAL = {Bulletin of the Iranian Mathematical Society},
    VOLUME = {43},
      YEAR = {2017},
    NUMBER = {1},
     PAGES = {217--221},
      ISSN = {1017-060X,1735-8515},
   MRCLASS = {05C50 (05C25 15A27)},
  MRNUMBER = {3622367},
MRREVIEWER = {Bassam\ H.\ Mourad},
}
\Addresses
\end{document}